\documentclass[11pt,reqno]{amsart}

\usepackage{floatrow}

\usepackage{amsmath}
\usepackage{amssymb}
\usepackage[left=1in,top=1in,right=1in,bottom=1in]{geometry}
\usepackage{epsfig}
\usepackage[normalem]{ulem}
\usepackage[usenames,dvipsnames]{color}
\usepackage{todonotes}
\usepackage[colorlinks, citecolor=black, linkcolor=black, urlcolor=black]{hyperref}
\usepackage{bm}
\usepackage{algorithm}
\usepackage{algpseudocode}

\usepackage{enumitem}

\allowdisplaybreaks

\usepackage{hyperref}
\usepackage{cleveref}
\crefname{subsection}{subsection}{subsections}

\usepackage{tikz}
\usetikzlibrary{backgrounds}
\usetikzlibrary{intersections}
\newtheorem{theorem}{Theorem}[section]
\newtheorem{lemma}[theorem]{Lemma}

\newcommand\eps{\varepsilon}
\newcommand{\E}{\mathbb E}

\newcommand{\Prob}{\mathbb{P}}
\newcommand{\Bin}{\mathrm{Bin}}

\newcommand{\Nn}{{\mathbb N}}

\newcommand{\scr}{\mathcal}
\newcommand{\mb}{\mathbb}

\theoremstyle{definition}
\newtheorem{remark}[theorem]{Remark}

\newcommand{\bigo}{\mathcal{O}}
\newcommand{\htau}{\hat{\tau}}

\title[Power of $k$ Choices in the Semi-Random Graph Process]{Power of $k$ Choices in the Semi-Random Graph Process}

\author{Pawe\l{} Pra\l{}at}
\address{Department of Mathematics, Toronto Metropolitan University, Toronto, Canada}
\email{pralat@torontomu.ca}

\author{Harjas Singh}
\address{Department of Mathematics, Toronto Metropolitan University, Toronto, Canada}
\email{harjas.singh@torontomu.ca}

\date{}

\begin{document}

\maketitle

\begin{abstract}
The semi-random graph process is a single player game in which the player is initially presented an empty graph on $n$ vertices. In each round, a vertex $u$ is presented to the player independently and uniformly at random. The player then adaptively selects a vertex $v$, and adds the edge $uv$ to the graph. For a fixed monotone graph property, the objective of the player is to force the graph to satisfy this property with high probability in as few rounds as possible.

In this paper, we introduce a natural generalization of this game in which $k$ random vertices $u_1, \ldots, u_k$ are presented to the player in each round. She needs to select one of the presented vertices and connect to any vertex she wants. We focus on the following three monotone properties: minimum degree at least $\ell$, the existence of a perfect matching, and the existence of a Hamiltonian cycle.
\end{abstract}

%%%%%%%%%%%%%%%%%%%%%%%%%%%%%%%%%%%%%%%%%%%%%%%%%
\section{Introduction and Main Results}
%%%%%%%%%%%%%%%%%%%%%%%%%%%%%%%%%%%%%%%%%%%%%%%%%
\subsection{Definitions} 

In this paper, we consider a natural generalization of the \textbf{semi-random graph process} suggested by Peleg Michaeli, introduced formally in~\cite{beneliezer2019semirandom}, and studied recently in~\cite{beneliezer2020fast,gao2020hamilton,gao2022perfect,behague2022,Burova2022,macrury_2022,gao2022fully,frieze2022hamilton,koerts2022k,gamarnik2023cliques} that can be viewed as a ``one player game''. The original process starts from $G_0$, the empty graph on the vertex set $[n]:=\{1,\ldots,n\}$ where $n \ge 1$. In each \textbf{step} $t$, a vertex $u_t$ is chosen uniformly at random from $[n]$. Then, the player (who is aware of the graph $G_t$ and the vertex $u_t$) must select a vertex $v_t$ and add the edge $u_tv_t$ to $G_t$ to form $G_{t+1}$. The goal of the player is to build a (multi)graph satisfying a given property $\scr{P}$ as quickly as possible. It is convenient to refer to $u_t$ as a {\bf square}, and $v_t$ as a {\bf circle} so every edge in $G_t$ joins a square with a circle. We say that $v_t$ is paired to $u_t$ in step $t$. Moreover, we say that a vertex $x \in [n]$ is \textbf{covered} by the square $u_t$ arriving at round $t$, provided $u_t = x$. The analogous definition extends to the circle $v_t$. Equivalently, we may view $G_t$ as a directed graph where each arc  directs from $u_t$ to $v_t$, and thus we may use $(u_t,v_t)$ to denote the edge added in step $t$. For this paper, it is easier to consider squares and circles for counting arguments.

We generalize the process as follows. Let $k \in \Nn$. In each step $t$, $k$ vertices $u^1_t, \ldots, u^k_t$ are chosen independently and uniformly at random from $[n]$. For simplicity, we allow repetitions  but, of course there will not be too many of them. Then, the player must select one of them (that is, select $i_t \in [k]$ and fix $u_t=u^{i_t}_t$), select a vertex $v_t$, and add the edge $u_t v_t$ to $G_t$ to form $G_{t+1}$. The objective is the same as before, namely, to achieve a property $\scr{P}$ as quickly as possible. We will refer to this game as the \textbf{$k$-semi-random graph process} or simply the \textbf{$k$-process}. Clearly, it is a generalization, since for $k=1$ we recover the original game. Moreover, if $k_1 > k_2 \ge 1$, then the $k_1$-process is as easy for the player as the $k_2$-process since additional $k_1-k_2$ squares may be simply ignored when making choices. However, it is often the case that more choices provide substantially more power to the player. For more details, we refer the reader to a general survey~\cite{survey_load_balancing} and the first paper introducing this powerful and fundamental idea~\cite{load_balancing}.

\medskip

A \textbf{strategy} $\scr{S}$ to play against the $k$-process is defined by specifying for each $n \ge 1$, a sequence of functions $(f_{t})_{t=1}^{\infty}$, where for each $t \in \mb{N}$, $f_t(u_1,v_1,\ldots, u_{t-1},v_{t-1},u^1_t, \ldots, u^k_t)$ is a distribution over $[k] \times [n]$ which depends on the vertices $u^1_t, \ldots, u^k_t$, and the history of the process up until step $t-1$. Then, $i_t \in [k]$ and $v_t$ is chosen according to this distribution. Observe that this means that the player needs to select her strategy (possibly randomized) in advance, before the game actually starts. If $f_t$ is an atomic distribution, then the pair $(i_t,v_t)$ is determined by $u_1,v_1, \ldots ,u_{t-1},v_{t-1},u^1_t, \ldots, u^k_t$. We then denote $(G_{i}^{\scr{S}}(n))_{i=0}^{t}$ as the sequence of random (multi)graphs obtained by following the strategy $\scr{S}$ for $t$ rounds; where we shorten $G_{t}^{\scr{S}}(n)$ to $G_{t}(n)$ or $G_t$ when clear. 

Suppose $\scr{P}$ is a monotonely increasing property. Given a strategy $\scr{S}$ to play against the $k$-process and a constant $0<q<1$, let $\htau_{\scr{P}}(\scr{S},q,n,k)$ be the minimum $t \ge 0$ for which $\mb{P}[G_{t}^{\scr{S}}(n) \in \scr{P}] \ge q$,
where $\htau_{\scr{P}}(\scr{S},q,n,k):= \infty$ if no such $t$ exists. Define
\[
\htau_{\scr{P}}(q,n,k) = \inf_{ \scr{S}} \htau_{\scr{P}}( \scr{S},q,n,k),
\]
where the infimum is over all strategies on $[k] \times [n]$. Observe that for each $n \ge 1$, if $0 \le q_{1} \le q_{2} \le 1$, then $\htau_{\scr{P}}(q_1,n,k) \le \htau_{\scr{P}}(q_2,n,k) $ as $\scr{P}$ is increasing. Thus, the function $q\rightarrow \limsup_{n\to\infty} \htau_{\scr{P}}(q,n,k)$ is non-decreasing, and so the limit
\[
\tau_{\scr{P}}(k):=\lim_{q\to 1^-}\limsup_{n\to\infty} \frac{\htau_{\scr{P}}(q,n,k) }{n},
\]
is guaranteed to exist. The goal is typically to compute upper and lower bounds on $\tau_{\scr{P}}(k)$ for various properties $\scr{P}$. Note that we normalized $\htau_{\scr{P}}(q,n,k)$ by $n$ above since the properties investigated in this paper need a linear number of rounds to be achieved. Other properties might require different scaling. For example, creating a fixed graph $H$ requires $o(n)$ rounds a.a.s.~\cite{beneliezer2019semirandom,behague2022}.

\subsection{Main Results} 

In this paper, we investigate the following three monotone properties: minimum degree at least $\ell$ (Section~\ref{sec:degree}), the existence of a perfect matching (Section~\ref{sec:matching}), and the existence of a Hamiltonian cycle (Section~\ref{sec:cycles}). 

For minimum degree at least $\ell$, we first show that a greedy strategy of choosing a minimum-degree vertex from the $k$ offered, and joining it to a minimum-degree vertex, is optimal; the proof adapts one from~\cite{beneliezer2019semirandom}. For fixed $k$ and $\ell$, the time to achieve minimum degree $\ell$ follows from the differential equation method.

For perfect matching, an optimal algorithm is not clear so we only provide some lower and upper bounds. The lower bound is the time to create minimum degree 1 graph, drawn from the earlier part. For fixed $k$, the upper bound comes from the differential-equation analysis of an algorithm that is a straightforward extension of that in~\cite{gao2022perfect}. 

Finally, for Hamiltonicity, the picture is similar to that of perfect matching. The lower bound is that for minimum degree 2, drawn from the earlier part. For fixed $k$, the upper bound comes again from the differential-equation analysis of a natural extension of the algorithm from~\cite{frieze2022hamilton}.

The computations presented in the paper (see Tables~\ref{tab:min_degree}, \ref{tab:perfect_matching}, and~\ref{tab:cycles}) were performed by using Maple~\cite{bernardin2016maple}. The worksheets can be found on-line\footnote{\url{https://math.torontomu.ca/~pralat/}}.

%%%%%%%%%%%%%%%%%%%%%%%%%%%%%%%%%%%%%%%%%%%%%%%%%
\section{Preliminaries}
%%%%%%%%%%%%%%%%%%%%%%%%%%%%%%%%%%%%%%%%%%%%%%%%%

%%%%%%%%%%%%%%%%%%%%%%%%%%%%%%%%%%%%%%%%%%%%%%%%%
\subsection{Notation}
%%%%%%%%%%%%%%%%%%%%%%%%%%%%%%%%%%%%%%%%%%%%%%%%%

The results presented in this paper are asymptotic by nature. We say that some property $\scr{P}$ holds \emph{asymptotically almost surely} (or a.a.s.) if the probability that the $k$-process has this property (after possibly applying some given strategy) tends to $1$ as $n$ goes to infinity. 
Given two functions $f=f(n)$ and $g=g(n)$, we will write $f(n)=\bigo(g(n))$ if there exists an absolute constant $c > 0$ such that $|f(n)| \leq c|g(n)|$ for all $n$, $f(n)=\Omega(g(n))$ if $g(n)=\bigo(f(n))$, $f(n)=\Theta(g(n))$ if $f(n)=\bigo(g(n))$ and $f(n)=\Omega(g(n))$, and we write $f(n)=o(g(n))$ or $f(n) \ll g(n)$ if $\lim_{n\to\infty} f(n)/g(n)=0$. In addition, we write $f(n) \gg g(n)$ if $g(n)=o(f(n))$ and we write $f(n) \sim g(n)$ if $f(n)=(1+o(1))g(n)$, that is, $\lim_{n\to\infty} f(n)/g(n)=1$.

We will use $\log n$ to denote a natural logarithm of $n$. As mentioned earlier, for a given $n \in \Nn := \{1, 2, \ldots \}$, we will use $[n]$ to denote the set consisting of the first $n$ natural numbers, that is, $[n] := \{1, 2, \ldots, n\}$. Finally, as typical in the field of random graphs, for expressions that clearly have to be an integer, we round up or down but do not specify which: the choice of which does not affect the argument.

%%%%%%%%%%%%%%%%%%%%%%%%%%%%%%%%
\subsection{Concentration Tools}\label{sec:concentration}
%%%%%%%%%%%%%%%%%%%%%%%%%%%%%%%%

Let us first state a few specific instances of Chernoff's bound that we will find useful. Let $X \in \textrm{Bin}(n,p)$ be a random variable distributed according to a Binomial distribution with parameters $n$ and $p$. Then, a consequence of \emph{Chernoff's bound} (see e.g.~\cite[Theorem~2.1]{JLR}) is that for any $t \ge 0$ we have
\begin{eqnarray}
\Prob( X \ge \E X + t ) &\le& \exp \left( - \frac {t^2}{2 (\E X + t/3)} \right)  \label{chern1} \\
\Prob( X \le \E X - t ) &\le& \exp \left( - \frac {t^2}{2 \E X} \right).\label{chern}
\end{eqnarray}

Moreover, let us mention that the bound holds in a more general setting as well, that is, for $X=\sum_{i=1}^n X_i$ where $(X_i)_{1\le i\le n}$ are independent variables and for every $i \in [n]$ we have $X_i \in \textrm{Bernoulli}(p_i)$ with (possibly) different $p_i$-s (again, see~e.g.~\cite{JLR} for more details). Finally, it is well-known that the Chernoff bound also applies to negatively correlated Bernoulli random variables~\cite{dubhashi1998balls}.

%%%%%%%%%%%%%%%%%%%%%%%%%%%%%%%%%%%%%%%%%%%%%%%%%
\subsection{The Differential Equation Method}
%%%%%%%%%%%%%%%%%%%%%%%%%%%%%%%%%%%%%%%%%%%%%%%%%

In this section, we provide a self-contained \textit{non-asymptotic} statement of the differential equation method which we will use for each property we investigate. The statement combines~\cite[Theorem $2$]{warnke2019wormalds}, and its extension~\cite[Lemma $9$]{warnke2019wormalds}, in a form convenient for our purposes, where we modify the notation of~\cite{warnke2019wormalds} slightly. In particular, we rewrite~\cite[Lemma $9$]{warnke2019wormalds} in a less general form in terms of a stopping time $T$. We need only check the `Boundedness Hypothesis' (see below) for $0 \le t \le T$, which is exactly the setting in our proofs.

Suppose we are given integers $a,n \ge 1$, a bounded domain $\scr{D} \subseteq \mb{R}^{a+1}$, and functions $(F_k)_{1 \le k \le a}$ where each $F_k: \scr{D} \to \mb{R}$ is $L$-Lipschitz-continuous on $\scr{D}$ for $L \ge 0$. Moreover, suppose that $R \in [1, \infty)$ and $S \in (0, \infty)$ are \textit{any} constants which satisfy $\max_{1 \le k \le a} |F_{k}(x)| \le R$ for all $x=(s,y_1,\ldots ,y_{a})\in \scr{D}$ and $0 \le s \le S$.

\begin{theorem}[Differential Equation Method, \cite{warnke2019wormalds}] \label{thm:differential_equation_method}
Suppose we are given $\sigma$-fields $\scr{F}_{0}  \subseteq \scr{F}_{1} \subseteq \cdots$, and for each $t \ge 0$, random variables $(Y_{k}(t))_{1 \le k \le a}$ which are $\scr{F}_t$-measurable. Define $T_{\scr{D}}$ to be the minimum $t \ge 0$ such that
\[
 (t/n, Y_{1}(t)/n, \ldots , Y_{a}(t)/n) \notin \scr{D}.
\]
Let $T \ge 0$ be an (arbitrary) stopping time\footnote{The stopping time $T\ge 0$ is \textbf{adapted} to $(\scr{F}_t)_{t \ge 0}$, provided the event $\{T = t\}$ is $\scr{F}_t$-measurable for each $t \ge 0$.} adapted to $(\scr{F}_t)_{t \ge 0}$, and assume that the following conditions hold for $\delta, \beta, \gamma \ge 0$ and $\lambda \ge \delta \min\{S, L^{-1}\} + R/n$:
\begin{enumerate}
    \item[(i)] The `Initial Condition': For some $(0,\hat{y}_1,\ldots ,\hat{y}_a) \in \scr{D}$, \label{enum:initial_conditions}
    \[
    \max_{1 \le k \le a} |Y_{k}(0) - \hat{y}_k n| \le \lambda n.
    \] 
    \item[(ii)] The `Trend Hypothesis': For each $k \in [a]$ and each  $t \le \min\{ T, T_{\scr{D}} -1\}$, \label{enum:trend_hypothesis}
    $$|\mb{E}[ Y_{k}(t+1) - Y_{k}(t) \mid \scr{F}_t] - F_{k}(t/n,Y_{1}(t)/n,\ldots ,Y_{a}(t)/n)| \le \delta.$$
    \item[(iii)] The `Boundedness Hypothesis': With probability $1 - \gamma$, \label{enum:boundedness_hypothesis}
    $$|Y_{k}(t+1) -  Y_{k}(t)| \le \beta,$$
    for each $k \in [a]$ and each $t \le \min\{ T, T_{\scr{D}} -1\}$.
\end{enumerate}
Then, with probability at least $1 - 2a \exp\left(\frac{-n \lambda^2}{8 S \beta^2}\right) - \gamma$, we have that
\begin{equation}
    \max_{0 \le t \le \min\{T, \sigma n\}} \max_{1 \le k \le a} |Y_{k}(t) -y_{k}(t/n) n| < 3 \lambda \exp(L S)n,
\end{equation}
where $(y_{k}(s))_{1 \le k \le a}$ is the unique solution to the system of differential equations
\begin{equation} \label{eqn:general_de_system}
    y_{k}'(s) = F_{k}(s, y_{1}(s),\ldots ,y_{a}(s)) \quad \mbox{with $y_{k}(0) = \hat{y}_k$ for $1 \le k \le a$,}
\end{equation}
and $\sigma = \sigma(\hat{y}_1,\ldots ,\hat{y}_a) \in [0,S]$ is any choice of $\sigma \ge 0$ with the property that $(s,y_{1}(s),\ldots, y_{a}(s))$ has $\ell^{\infty}$-distance at least $3 \lambda \exp(LS)$ from the boundary of $\scr{D}$ for all $s \in [0, \sigma)$.
\begin{remark}
Standard results for differential equations guarantee that \eqref{eqn:general_de_system} has a unique solution $(y_{k}(s))_{1\le k \le a}$ which extends arbitrarily close to the boundary of $\scr{D}$. 
\end{remark}
\end{theorem}

%%%%%%%%%%%%%%%%%%%%%%%%%%%%%%%%%%%%%%%%%%%%%%%%%
\section{Minimum Degree at Least $\ell$}\label{sec:degree}
%%%%%%%%%%%%%%%%%%%%%%%%%%%%%%%%%%%%%%%%%%%%%%%%%

Let us fix a natural number $\ell$. Our goal is to investigate how long does it take for the $k$-process to create a graph with minimum degree at least $\ell$. This problem was considered in~\cite{beneliezer2019semirandom} for the original semi-random process ($k=1$). In this paper, we investigate it for the $k$-process for any value of $k$.

Let $\scr{P}_\ell$ be the property that a graph has a minimum degree at least $\ell$. In order to establish the value of $\tau_{\scr{P}_\ell}(k)$, we need to do two things: investigate which strategy is optimal (we do it in Subsection~\ref{sec:min_degree_optimal}) and then analyze the optimal strategy (we do it in Subsection~\ref{sec:min_degree_analysis}). One consequence of our results is Table~\ref{tab:min_degree} which consists of numerical values of $\tau_{\scr{P}_\ell}(k)$ for a grid of parameters $(k,\ell)$ with $1 \le k,\ell \le 5$. It follows immediately from the definition of $\tau_{\scr{P}_\ell}(k)$ that  it is a non-decreasing function with respect to $\ell$ but a non-increasing one with respect to $k$. 

Finally, we note that for large values of $k$ (and any fixed value of $\ell$), typically some square lands on a vertex with minimum degree and so the degree distribution is well balanced during the whole process. As a result, the total number of rounds is close to the trivial lower bound of $\ell n/2$. In other words, $\tau_{\scr{P}_\ell}(k) = \ell/2 + o_k(1)$. Similarly, for large values of $\ell$ (and any fixed value of $k$), because of the law of large numbers, each vertex receives more or less the same number of squares. As before, the degree distribution is well balanced and as a consequence, $\tau_{\scr{P}_\ell}(k) = \ell/2 + o_\ell(1)$. We investigate both of these properties in Subsection~\ref{sec:min_degree_large}.

\begin{table}[htp]
\caption{Minimum Degree at Least $\ell$---numerical values of $\tau_{\scr{P}_\ell}(k)$ for a grid of parameters $(k,\ell)$ with $1 \le k,\ell \le 5$.}
\begin{center}
\begin{tabular}{|c|c|c|c|c|c|c|}
\hline
& $k=1$ & $k=2$ & $k=3$ & $k=4$ & $k=5$ \\
\hline
$\ell=1$ & 0.69315 & 0.62323 & 0.59072 & 0.57183 & 0.55947 \\
$\ell=2$ & 1.21974 & 1.12498 & 1.09081 & 1.07184 & 1.05947 \\
$\ell=3$ & 1.73164 & 1.62508 & 1.59081 & 1.57184 & 1.55947 \\
$\ell=4$ & 2.23812 & 2.12508 & 2.09081 & 2.07184 & 2.05947 \\
$\ell=5$ & 2.74200 & 2.62508 & 2.59081 & 2.57184 & 2.55947 \\
\hline
\end{tabular}
\end{center}
\label{tab:min_degree}
\end{table}

%%%%%%%%%%%%%%%%%%%%%%%%%%%%%%%%%%%%%%%%%%%%%%%%%
\subsection{Optimal Strategy}\label{sec:min_degree_optimal}

In this subsection, we show that the following greedy strategy is an optimal strategy. In this strategy, in each round $t$ of the process the player selects a square that lands on a vertex with the smallest degree, that is, she selects $i_t$ such that $\deg_{G_t}(u_t^i) \ge \deg_{G_t}(u_t^{i_t})$ for any $i \in [k]$; if there is more than one such square to choose from, the decision which one to select is made arbitrarily. Then, the player puts a circle on a vertex with minimum degree; again, if there is more than one such vertex to choose from, the decision is made arbitrarily. Let us denote this strategy as $\scr{S}_0$.

Recall that, in order to make sure the $k$-process is well defined, we allowed it to create multi-graphs (that is, loops and parallel edges are allowed). It simplifies the definition of the strategy $\scr{S}_0$ above and our proofs below but one can easily adjust the argument to stay with the family of simple graphs.

Let us fix $k$ and $\ell$. For a given strategy $\scr{S}$, let $H(\scr{S})$ be the hitting time for the property $\scr{P}_\ell$, that is, $H(\scr{S})$ is the random variable equal to the number of rounds required for $\scr{S}$ to achieve the property $\scr{P}_\ell$. We say that a strategy $\scr{S}$ \textbf{dominates} a strategy $\scr{S}'$ if the random variable $H(\scr{S}')$ is dominated by the random variable $H(\scr{S})$, that is, $\Prob( H(\scr{S}) \le t) \ge \Prob( H(\scr{S}') \le t)$ for any $t$.

The next lemma is straightforward. Its proof is an adaptation of the proof for the original process~\cite{beneliezer2019semirandom}. 

\begin{lemma}\label{lem:optimalstrategy}
Let $k, \ell \in \Nn$, and consider the property $\scr{P}_\ell$. The strategy $\scr{S}_0$ dominates any other strategy $\scr{S}$ against the $k$-process.
\end{lemma}
\begin{proof}
We say that a strategy $\scr{S}$ is \textbf{$(i,j)$-minimizing} if in each of the first $i$ rounds, the player chooses a square of minimum degree and in each of the first $j$ rounds, she puts a circle on a vertex of minimum degree. Strategy $\scr{S}$ is said to be \textbf{minimizing} if it is $(i,j)$-minimizing for every $i$ and $j$. In order to prove the lemma, since domination is a transitive relation and any strategy is $(0,0)$-minimizing, it is enough to show that any $(i,j)$-minimizing strategy is dominated by some $(i+1,j)$-minimizing strategy as well as some $(i,j+1)$-minimizing strategy. Since $\scr{S}_0$ is minimizing and any two minimizing strategies dominate one another, we conclude that $\scr{S}_0$ dominates any other strategy $\scr{S}$.

Consider any $(i,j)$-minimizing strategy $\scr{S}$. We will modify it slightly and create two new strategies, $\scr{S}'$ and $\scr{S}''$, that are $(i+1,j)$-minimizing and, respectively, $(i,j+1)$-minimizing. We can imagine a player using strategy $\scr{S}$ on the graph $G_t$ and another player using strategy $\scr{S}'$ on an auxiliary graph $G'_t$. The two games are coupled such that squares appear in both games at the same locations. 

During the first $i$ rounds, the strategy $\scr{S}'$ is the same as the strategy $\scr{S}$. Suppose that at round $i+1$, with probability $p>0$, $\scr{S}$ chooses a square that lands on a vertex $v$ but the decision is made in a non-greedy fashion (that is, the chosen vertex does not have minimum degree among the $k$ offered vertices). We condition on this event, and slightly modify $\scr{S}$ to get $\scr{S}'$ as follows. At round $i+1$, strategy $\scr{S}'$ chooses a square that lands on a vertex $u$ that has minimum degree; in particular, $\deg_{G_{i}}(u) < \deg_{G_{i}}(v)$. From that point on, the two graphs, $G_t$ and $G'_t$, are going to differ. For the rest of the game, as long as $\deg_{G_{t}}(u) \le \deg_{G_{t}}(v)-2$, $\scr{S}'$ continues ``stealing'' strategy $\scr{S}$, that is, both the choices for squares and for circles are exactly the same in both games. If, at any point of the game, $\deg_{G_{t}}(u) = \deg_{G_{t}}(v)-1$, then $u$ and $v$ are ``relabeled'' in $G'_t$ (that is, $v$ becomes $u$ and $u$ becomes $v$). After that, we continue coupling the two games but now each time a square lands on $u$ from $G_t$, we modify the coupling so that it also lands on $u$ (that used to be initially labelled as $v$) in $G'_t$. The same property holds for $v$. Clearly, after this modification, we preserve the property that vertices $u^1_t, \ldots, u^k_t$ are chosen independently and uniformly at random from $[n]$. The strategy $\scr{S}'$ continues ``stealing'' strategy $\scr{S}$. 

A simple but important property is that from time $i+1$ on, but before possible relabelling, $\deg_{G'_t} (u) = \deg_{G_t}(u)+1$ and $\deg_{G'_t}(v) = \deg_{G_t}(v)-1$. Since $\deg_{G_{t}}(u) \le \deg_{G_{t}}(v)-2$, $\deg_{G'_{t}}(u) \le \deg_{G'_{t}}(v)$. As a consequence, 
$$
\min\{ \deg_{G_t}(u), \deg_{G_t}(v) \} = \deg_{G_t}(u) = \deg_{G'_t}(u)-1 = \min\{ \deg_{G'_t}(u), \deg_{G'_t}(v) \} - 1.
$$
For any other vertex $w \not\in \{u, v\}$, $\deg_{G_t}(w) = \deg_{G'_t}(w)$. Hence, provided that no relabelling took place, $\min\{ \deg_{G_t}(u), \deg_{G_t}(v) \} \ge \ell$ implies that $\min\{ \deg_{G'_t}(u), \deg_{G'_t}(v) \} \ge \ell+1$ and so the desired property $\scr{P}_\ell$ cannot be achieved by the strategy $\scr{S}$ before it is achieved by the strategy $\scr{S}'$. Finally, note that when $\deg_{G_{t}}(u) = \deg_{G_{t}}(v)-1$, we have that $\deg_{G_{t}}(u) = \deg_{G'_{t}}(v)$ and $\deg_{G_{t}}(v) = \deg_{G'_{t}}(u)$. Hence, after relabelling, the degree distribution in $G_t$ is exactly the same as the degree distribution in $G'_t$ (despite the fact that graphs are possibly different). This property will be preserved to the end of the process and so both strategies will achieve the desired property $\scr{P}_\ell$ at the same time. 

The same argument can be repeated to create an $(i,j+1)$-minimizing strategy $\scr{S}''$. This finishes the proof of the lemma.
\end{proof}

%%%%%%%%%%%%%%%%%%%%%%%%%%%%%%%%%%%%%%%%%%%%%%%%%
\subsection{Analysis of the Optimal Strategy}\label{sec:min_degree_analysis}

In this subsection, we analyze the greedy strategy that was introduced and proved to be optimal in the previous subsection. This establishes $\tau_{\scr{P}_\ell}(k)$ for any value of $\ell$ and $k$. 

\begin{theorem}
Let $k, \ell \in \Nn$. Then, $\tau_{\scr{P}_\ell}(k) = c_{\ell,k}$, where $c_{\ell,k} \ge \ell/2$ is a constant that is derived from a system of differential equations. The numerical values for $1 \le k,\ell \le 5$ are presented in Table~\ref{tab:min_degree}.
\end{theorem}

\begin{proof}
In the greedy strategy $\scr{S}_0$, we distinguish phases by labelling them with integers $q \in \{0, 1, \ldots, \ell-1\}$. During the $q$th phase, the minimum degree in $G_t$ is equal to $q$. In order to analyze the evolution of the $k$-process, we will track the following sequence of $\ell$ variables: for $0 \le i \le \ell-1$, let $Y_i = Y_i(t)$ denote the number of vertices in $G_t$ of degree $i$. 

Phase~$0$ starts at the beginning of the $k$-process. Since $G_0$ is empty, $Y_0(0) = n$ and $Y_i(0) = 0$ for $1 \le i \le \ell-1$. There are initially many isolated vertices but they quickly disappear. Phase~$0$ ends at time $t$ which is the smallest value of $t$ for which $Y_0(t) = 0$. The DEs method will be used to show that a.a.s.\ Phase~$0$ ends at time $t_0 \sim x_0 n$, where $x_0$ is an explicit constant which will be obtained by investigating the associated system of DEs. Moreover, the number of vertices of degree $i$ ($1 \le i \le \ell-1$) at the end of this phase is well concentrated around some values that are also determined based on the solution to the same system of DEs: a.a.s.\ $Y_i(t_0) \sim y_i(x_0) n$. With that knowledge, we move on to Phase~$1$ in which we prioritize vertices of degree 1. 

Consider any Phase~$q$, where $q \in \{0, 1, \ldots, \ell-1\}$. This phase starts at time $t_{q-1}$, exactly when the previous phase ends (or at time $t_{-1} := 0$ if $q=0$). At that point, the minimum degree of $G_{t_{q-1}}$ is $q$, so $Y_i(t) = 0$ for any $t \ge t_{q-1}$ and $i < q$. Hence, we only need to track the behaviour of the remaining $\ell-q$ variables. Let $\mathcal{A}_j(t)$ be the event that at time $t$ the player selects a square that lands on a vertex with degree $j$, that is, $\mathcal{A}_j(t) = \{ \deg_{G_t}(u_t^{i_t}) = j \}$. The probability that $\mathcal{A}_j(t)$ holds can be computed based on the sequence of $Y_i(t)$'s. To that end, it is convenient to introduce the following auxiliary event. Let $\mathcal{B}_j(t)$ be the event that at time $t$ all squares land on vertices of degree at least $j$, that is, $\mathcal{B}_j(t) = \{ \deg_{G_t}(u_t^{i}) \ge j, \text{ for all } i \in [k] \}$. Clearly, $\mathcal{B}_{j+1}(t) \subseteq \mathcal{B}_j(t)$ and $\mathcal{A}_j(t) = \mathcal{B}_j(t) \setminus \mathcal{B}_{j+1}(t)$. As a result,
\begin{align*}
\Prob ( \mathcal{A}_j(t) ) &= \Prob( \mathcal{B}_j(t) ) - \Prob ( \mathcal{B}_{j+1}(t) ) \\
&= \left( 1 - \sum_{a = q}^{j-1} \frac {Y_a(t)}{n} \right)^k - \left( 1 - \sum_{a = q}^{j} \frac {Y_a(t)}{n} \right)^k.
\end{align*}
Let us denote $H(t) = (Y_q(t), Y_{q+1}(t), \ldots, Y_{\ell-1}(t))$. Let $\delta_A$ be the Kronecker delta for the event $A$, that is, $\delta_A = 1$ if $A$ holds and $\delta_A=0$ otherwise. Then, for any $i$ such that $q \le i \le \ell-1$, 
\begin{align}\label{eq:min_degree_trend} % `Trend Hypothesis' 
\E ( Y_i(t+1) - Y_i(t) ~|~ H(t) ) &= -\delta_{i=q} + \delta_{i=q+1} - \Prob( \mathcal{A}_i(t) ) + \delta_{i \ge q+1} \Prob( \mathcal{A}_{i-1}(t) ) + \bigo(1/n).
\end{align}
Indeed, since the circle is put on a vertex of degree $q$, we always lose one vertex of degree $q$ (term $-\delta_{i=q}$) that becomes of degree $q+1$ (term $\delta_{i=q+1}$). The term $\bigo(1/n)$ is added to cover the (rare) case when one of the squares lands on the very last vertex of degree $q$. Alternatively, we could have stopped analyzing the duration of Phase $q$ prematurely when the number of vertices of degree $q$ is at most one; this phase would finish it in at most one extra round anyway. We might lose a vertex of degree $i$ when the selected square lands on a vertex of degree $i$ (term $\Prob( \mathcal{A}_i(t) )$. We might also gain one of them when the selected square lands on a vertex of degree $i-1$ (term, $\Prob( \mathcal{A}_{i-1}(t)$); note that this is impossible if $i = q$ (term $\delta_{i \ge q+1}$). This suggests the following system of DEs: for any $i$ such that $q \le i \le \ell-1$, 
\begin{align}
y'_i(x) &= -\delta_{i=q} + \delta_{i=q+1} \nonumber \\
&- \left( \left( 1 - \sum_{a = q}^{i-1} y_a(x) \right)^k - \left( 1 - \sum_{a = q}^{i} y_a(x) \right)^k \right)  \nonumber \\
&+ \delta_{i \ge q+1} \left( \left( 1 - \sum_{a = q}^{i-2} y_a(x) \right)^k - \left( 1 - \sum_{a = q}^{i-1} y_a(x) \right)^k \right). \label{eq:DEs_min_degree}
\end{align}

Let us now check that the assumptions of the DEs method are satisfied and then discuss the conclusions. Let $\eps > 0$ be an arbitrarily small constant and $\omega = \omega(n)$ be any function that tends to infinity as $n \to \infty$. We will ensure that for some universal constant $C > 0$, at the beginning of Phase~$q$, the initial condition is satisfied with $\lambda = C^{q} \omega / \sqrt{n} = o(1)$. (At the beginning of Phase~0, there is no error in the initial condition so this property is trivially satisfied.) In particular, we assume that the phase starts at time $t_{q-1} \sim x_{q-1} n$ for some constant $x_{q-1} \in [0,\infty)$, and for any $q \le i \le \ell-1$, $Y_i(t_{q-1}) \sim y_i(x_{q-1}) n$ for some constants $y_i(x_{q-1}) \in (0,1]$. The right hand side of~(\ref{eq:DEs_min_degree}) is continuous, bounded, and Lipschitz in the connected open set 
$$
\scr{D} = \{ (x, y_q, \ldots, y_{\ell-1}) : -\eps < x < \ell + \eps, -\eps < y_i < 1+\eps \},
$$
which contains the point $(x_{q-1}, y_q(x_{q-1}), \ldots, y_{\ell-1}(x_{q-1}) )$. Note that there is a small error in the `Trend Hypothesis', that is, $\delta = \bigo(1/n)$ (see~(\ref{eq:min_degree_trend})). Finally, note that the `Boundedness Hypothesis' holds deterministically ($\gamma = 0$) with $\beta=2$. 

We conclude, based on Theorem~\ref{thm:differential_equation_method}, that a.a.s.\ during the entire Phase~$q$,
$$
\max_{q \le i \le \ell-1} |Y_{i}(t) -y_{i}(t/n) n| < \lambda C n = o(n),
$$
provided that $C$ is a large enough constant. In particular, Phase~$q$ ends at time $t_q \sim x_q n$, where $x_q > x_{q-1}$ is the solution of the equation $y_q(x)=0$. Using the final values $y_i(x_q)$ in Phase~$q$ as initial values for Phase~$q+1$ we can repeat the argument inductively moving from phase to phase. The desired property is achieved at the end of Phase $\ell-1$ when a graph of minimum degree equal to $\ell$ is reached. 
\end{proof}

%%%%%%%%%%%%%%%%%%%%%%%%%%%%%%%%%%%%%%%%%%%%%%%%%
\subsection{Large value of $k$ or $\ell$}\label{sec:min_degree_large}

A natural question that arises is about the asymptotic behaviour of $\tau_{\scr{P}_\ell}(k)$ as either $\ell$ grows large or $k$ grows large.

\medskip

First, let us show that $\tau_{\scr{P}_\ell}(k) \to \ell/2$ as $\ell \to \infty$. 

\begin{theorem}
Fix $k \in \Nn$. Then, 
$$
\frac {\ell}{2} \le \tau_{\scr{P}_\ell}(k) \le \frac {\ell}{2} \left( 1 + \bigo(\sqrt{\log \ell / \ell}) \right).
$$
\end{theorem}

\begin{proof}
Noting that $\tau_{\scr{P}_\ell}(k)$ is a non-increasing function in $k$ and, trivially, for any value of $k$ we have $\tau_{\scr{P}_\ell}(k) \ge \ell/2$, we only investigate the case $k=1$.  One can try to analyze an optimal, greedy strategy but we aim for an easy argument without trying to optimize the error term, as long as it goes to zero as $\ell \to \infty$.  

Our algorithm consists of two phases. In Phase $1$, which lasts $\ell n / 2$ rounds, we place circles sequentially, that is, in round $i$, a circle is placed on vertex $i-1 \pmod{n} + 1$. As a result, at the end of Phase $1$, each vertex has exactly $\ell/2$ circles. Let $X_v$ denote the number of squares on vertex $v$ at the end of Phase~1. Then $X_v \in \Bin(\ell n/2, 1/n)$, with $\E(X_v) = \ell/2$. Let $t := 2\sqrt{\ell \log \ell}$. Then, by the Chernoff bound~(\ref{chern}) 
$$
  \Prob(X_v\leq \ell/2 - t ) \leq \exp \left( -\frac{t^2}{\ell} \right)  = \exp(-4 \log \ell) = 1 / \ell^4.
$$
Hence, we expect at most $n / \ell^4$ vertices with at most $\ell/2 - t$ squares. More importantly, the events ``$X_v\leq \ell/2 - t$'' associated with various vertices are negatively correlated. It follows immediately from the Chernoff bound~(\ref{chern1}) (see also the comment right after it) that a.a.s.\ there are at most $2n / \ell^4$ vertices with at most $\ell/2 - t$ squares. (Alternatively, one could estimate the variance and use Chebyshev's inequality.) 

Let a vertex $v$ be considered deficient if $\deg(v) < \ell$. Furthermore, define the deficit of a deficient vertex $v$ to be equal to $\ell - \deg(v)$. Then, at the end of Phase~$1$, a.a.s.\ at most $2n /\ell^4$ vertices have a deficit of at most $\ell/2$ (a trivial, deterministic upper bound), with the remaining vertices having deficit at most $2\sqrt{\ell \log \ell}$. In Phase $2$, we place circles on the deficient vertices to bring the deficit down to $0$. This takes at most $n/\ell^3 + 2n\sqrt{\ell \log \ell}$ rounds. Thus, the total number of rounds is at most 
$$
n \ell/2 + n/\ell^3 + 2n\sqrt{\ell \log \ell} = \frac {n\ell}{2} \left( 1 + \bigo(\sqrt{\log \ell / \ell}) \right).
$$ 
It follows that $\tau_{\scr{P}_\ell}(k) = \ell/2 +o_\ell(1) $ as $\ell \to \infty$, as required.
\end{proof}

\medskip

Next, let us show that $\tau_{\scr{P}_\ell}(k) \to \ell/2$ as $k \to \infty$. 

\begin{theorem}
Fix $\ell \in \Nn$. Then, 
$$
\frac {\ell}{2} \le \tau_{\scr{P}_\ell}(k) \le \frac {\ell}{2} \Big( 1 + \bigo( \log k / k ) \Big).
$$
\end{theorem}

\begin{proof}
We will investigate a greedy algorithm by considering $\ell$ phases. As before, we do not try to optimize the error term and aim for an easy argument. During Phase~$i$, the minimum degree is equal to $i-1$. The algorithm stops at the end of Phase~$\ell$. We will show that a.a.s.\ each phase takes at most $n/2 + n \log k / k$ rounds, so the total number of rounds is at most $(n\ell/2) (1 + 2 \log k / k)$. Since, trivially, $\tau_{\scr{P}_\ell}(k) \ge \ell/2$, we will get that $\tau_{\scr{P}_\ell}(k) = (\ell/2) ( 1 + \bigo( \log k / k) ) = \ell/2 +o_k(1) $ as $k \to \infty$.

Suppose that Phase~$i$ starts at time $t_i$. Let $X_t$ be the number of vertices of degree $i-1$ at the beginning of round $t$. Clearly, $X_{t_i} \le n$. It is convenient to consider two sub-phases. The first sub-phase continues as long as $X_t \ge n \log k / k$. Note that at any step $t$ of this sub-phase, the probability that no square lands on a vertex of degree $i-1$ is equal to
$$
(1-X_t/n)^k \le \exp( -kX_t/n) \le \exp( - \log k) = 1/k.
$$
It means that $X_t$ goes down by 1 with probability at most $1/k$ and goes down by 2, otherwise. In other words, at time $t \ge t_i$ during this sub-phase, the number of vertices of degree $i-1$ can be stochastically upper bounded as follows:
$$
X_{t} \le X_{t_0} - 2(t-t_i) + {\rm Bin}(t, 1/k) \le n - 2(t-t_i) + {\rm Bin}(t, 1/k).
$$
(The term ${\rm Bin}(t, 1/k)$ corresponds to the number of rounds during which the algorithm ``slows down'' because no square lands on a vertex of degree $i-1$.) Hence, the probability that the first sub-phase does not finish in less than $n/2$ rounds is at most
\begin{eqnarray*}
\Prob ( {\rm Bin}(n/2, 1/k) \ge n \log k / k ) &\le& \Prob ( {\rm Bin}(n/2, 1/k) \ge 2 \E ( {\rm Bin}(n/2, 1/k) ) ) \\
&\le& \exp( - \Theta( n ) ) = o(1),
\end{eqnarray*}
assuming that $k \ge 3 > e$ (which we may, since we aim for a result that holds for $k$ large enough). The second sub-phase takes at most $n \log k / k$ steps (deterministically) so, a.a.s.\ the entire phase ends in at most $n/2 + n \log k / k$ steps, and the desired property holds.
\end{proof}

%%%%%%%%%%%%%%%%%%%%%%%%%%%%%%%%%%%%%%%%%%%%%%%%%
\section{Perfect Matchings}\label{sec:matching}
%%%%%%%%%%%%%%%%%%%%%%%%%%%%%%%%%%%%%%%%%%%%%%%%%

In this section, we investigate another classical monotone property that was already studied in the context of semi-random processes, namely, the property of having a perfect matching, which we denote by {\tt PM}. In the very first paper~\cite{beneliezer2019semirandom}, it was shown that the semi-random process is general enough to approximate (using suitable strategies) several well-studied random graph models, including an extensively studied $\ell$-out process (see, for example, Chapter~18 in~\cite{Karonski_Frieze}). In the $\ell$-out process, each vertex independently connects to $\ell$ randomly selected vertices which results in a random graph on $n$ vertices and $\ell n$ edges. 

Since the $2$-out process has a perfect matching a.a.s.~\cite{frieze1986maximum}, we immediately get that $\tau_{\texttt{PM}}(k) \le \tau_{\texttt{PM}}(1) \le 2$. By coupling the semi-random process with another random graph that is known to have a perfect matching a.a.s.~\cite{pittel}, the bound can be improved to $1+2/e < 1.73576$. This bound was recently improved by investigating a fully adaptive algorithm~\cite{gao2022perfect}. The currently best upper bound is $\tau_{\texttt{PM}}(1) < 1.20524$ but there is an easy algorithm that yields the following bound: $\tau_{\texttt{PM}}(1) < 1.27695$. In this paper, we adjust the easy algorithm to deal with the $k$-process and present the corresponding upper bounds (see Subsection~\ref{sec:pm_upperbound}). One could adjust the more sophisticated algorithm as well. We do not do it as the improvement is less significant for larger values of $k$ but the argument is substantially more involved. 

Let us now move to the lower bounds. In the initial paper introducing the semi-random process~\cite{beneliezer2019semirandom}, it was already observed that $\tau_{\texttt{PM}}(1) \ge \tau_{\scr{P}_1}(1) = \ln(2) > 0.69314$. This lower bound was improved as well, and now we know that $\tau_{\texttt{PM}}(1) > 0.93261$~\cite{gao2022perfect}. Since adapting the argument from~\cite{gao2022perfect} to $k \ge 2$ would be much more involved and the improvement would be less significant, we only use the trivial bound and the results from the previous section: $\tau_{\texttt{PM}}(k) \ge \tau_{\scr{P}_1}(k)$. Indeed, the gap between the upper and lower bounds gets small as $k$ is large. In fact, $\tau_{\texttt{PM}}(k) \to 1/2$ as $k \to \infty$ (see Subsection~\ref{sec:pm_large}). 

\begin{table}[htp]
\caption{Perfect Matchings---numerical upper and lower bounds of $\tau_{\texttt{PM}}(k)$ for $1 \le k \le 10$. Stronger bounds for $k=1$ follow from~\cite{gao2022perfect}.}
\begin{center}
\begin{tabular}{|c|c|c|c|c|c|c|}
\hline
& lower bound & upper bound & & lower bound & upper bound \\
\hline
$k=1$ & 0.69315 (0.93261) & 1.27696 (1.20524) & $k=6$ & 0.55075 & 0.66425 \\
$k=2$ & 0.62323 & 0.92990 & $k=7$ & 0.54426 & 0.64243 \\
$k=3$ & 0.59072 & 0.80505 & $k=8$ & 0.53924 & 0.62573 \\
$k=4$ & 0.57183 & 0.73708 & $k=9$ & 0.53525 & 0.61255 \\
$k=5$ & 0.55947 & 0.69402 & $k=10$ & 0.53199 & 0.60187 \\
\hline
\end{tabular}
\end{center}
\label{tab:perfect_matching}
\end{table}

%%%%%%%%%%%%%%%%%%%%%%%%%%%%%%%%%%%%%%%%%%%%%%%%%
\subsection{Upper bound for $\tau_{\texttt{PM}}(k)$}\label{sec:pm_upperbound}

In this subsection, we analyze the following simple but fully adaptive strategy. In each step $t$ of the algorithm, we will track a partial matching $M_t$ that is already built and the set $U_t$ of unsaturated vertices. Initially, $M_0 = \emptyset$ and $U_0 = [n]$. We will use $V[M_t]$ to denote the set of vertices associated with the edges in $M_t$. Some vertices in $V[M_t]$ will be coloured red or green, and some edges (outside of $M_t$) will be coloured green. The colours will be used to extend a partial matching. Suppose that an edge $bc \in M_t$ and an edge $ab$ is green. (This will make vertex $b$ to be green and vertex $c$ to be red.) If, at some point of the process, the square lands on the red vertex $c$, then the player can create an augmenting path $abcd$ by adding an edge $cd$ to some vertex $d$ outside of $M_t$.

Suppose that at time $t$, $k$ squares land on vertices $u^1_t, \ldots, u^k_t$. We consider a few cases.
\begin{itemize}
\item [Case~(a):] At least one square lands on a vertex from $U_{t-1}$, that is, $\{ u^1_t, \ldots, u^k_t \} \cap U_{t-1} \neq \emptyset$. We arbitrarily select $u_t \in \{ u^1_t, \ldots, u^k_t \} \cap U_{t-1}$, and let $v_t$ be a uniformly random vertex in $U_{t-1}$. If $v_t = u_t$, then the matching and the set of unsaturated vertices do not change: $M_t=M_{t-1}$ and $U_t = U_{t-1}$. Otherwise, we extend the partial matching by adding an edge we just created to $M_{t-1}$, that is, $M_t = M_{t-1} \cup \{ u_t v_t \}$ and $U_t = U_{t-1} \setminus \{ u_t, v_t \}$. For every green vertex $x \in V[M_{t-1}]$, if it is adjacent to either $u_t$ or $v_t$ by a green edge, then we uncolour this green edge, uncolour $x$ (from green), and uncolour the mate of $x$ in $M_{t-1}$ (from red).
\item [Case~(b):] No square lands on a vertex from $U_{t-1}$ but at least one square lands on a red vertex in $V[M_{t-1}]$. We arbitrarily select one of such red vertices to be $u_t$, and let $v_t$ be a uniformly random vertex in $U_{t-1}$. Let $x \in V[M_{t-1}]$ be the mate of $u_t$ in $M_{t-1}$. Let $y$ be the (unique) vertex in $U_{t-1}$ which is adjacent to $x$ by a green edge. If $v_t = y$, then the matching and the set of unsaturated vertices do not change: $M_t=M_{t-1}$ and $U_t = U_{t-1}$. Otherwise, let $M_t$ be the matching obtained by augmenting along the path $yxu_tv_t$, that is, $M_t = (M_{t-1} \setminus \{ xu_t \}) \cup \{yx, u_tv_t\}$. Let $U_t = U_{t-1} \setminus \{y, v_t\}$. Finally, update the green vertices and edges and the red vertices accordingly as in Case~(a).
\item [Case~(c):] No square lands on a vertex from $U_{t-1}$ nor on a red vertex in $V[M_{t-1}]$ but at least one square lands on an uncoloured vertex in $V[M_{t-1}]$. We arbitrarily select one of such uncoloured vertices to be $u_t$, and let $v_t$ be a uniformly random vertex in $U_{t-1}$. Colour the edge $u_tv_t$ and the vertex $u_t$ green and colour the mate of $u_t$ in $M_{t-1}$ red. The matching is not affected, that is, $M_t = M_{t-1}$ and $U_t = U_{t-1}$.
\item [Case~(d):] All squares land on green vertices. Let $v_t$ be an arbitrary vertex in $[n]$. The edge $u_tv_t$ will not be used in the process of constructing a perfect matching. Let $M_t = M_{t-1}$ and $U_t = U_{t-1}$.
\end{itemize}
As it was done in~\cite{gao2022perfect}, we terminate the algorithm prematurely (in order to avoid technical issues with the DEs method that will be used) at the step when $|U_t|$ becomes at most $\eps n$ where $\eps = 10^{-14}$. To saturate the remaining unsaturated vertices, the clean-up algorithm can be used that was introduced and analyzed in~\cite{gao2022perfect}. This algorithm was used for the original semi-random process (for $k=1$) but, by monotonicity, also applies for any $k \in \Nn$. It is not as efficient as the one described above but it may be easily analyzed. It was proved in~\cite{gao2022perfect} that a.a.s.\ this algorithm takes at most $100\sqrt{\eps} n = 10^{-5}n$ steps, which is numerically insignificant. 

\begin{theorem}
Let $k \in \Nn$. Then, $\tau_{\texttt{PM}}(k) \le u_{k} + 10^{-5}$, where $u_{k} \ge 1/2$ is a constant that is derived from a system of differential equations. 
The numerical bounds for $1 \le k \le 10$ are presented in Table~\ref{tab:perfect_matching}.
\end{theorem}

\begin{proof}
To analyze the above algorithm, we introduce the following random variables. Let $X(t)$ be the number of saturated vertices, that is, $X(t) = |V(M_t)| = 2|M_t|$. Let $R(t)$ be the number of red vertices in $V[M_t]$. The algorithm is designed in such a way that $R(t)$ is also equal to the number of green vertices, and thus equal to the number of green edges. We will use the DEs method to analyze the behaviour of the sequence $H_t := (X(t), R(t))$ but we will not encompass the full history of the process. For convenience, we will condition on less information and do not reveal the placement of circles associated with green edges; their placements amongst the unsaturated vertices remain distributed uniformly at random. 

Let us start with analyzing $X(t)$. Let $\mathcal{A}^i_{t+1}$ be the event that Case~(i) occurred at step $t+1$. Note that at step $t$, the set of vertices is partitioned into unsaturated vertices, red vertices in $V[M_t]$, uncoloured vertices in $V[M_t]$, and green vertices in $V[M_t]$. The algorithm makes a greedy selection of squares from these classes. There are $X(t)$ vertices that are \emph{not} unsaturated, $X(t)-R(t)$ of them are \emph{not} red, and $R(t)$ of the remaining ones are \emph{not} uncoloured (that is, are green). It follows then that
\begin{eqnarray*}
\Prob( \mathcal{A}^a_{t+1} ) &=& 1 - \left( \frac {X(t)}{n} \right)^k \\
\Prob( \mathcal{A}^b_{t+1} ) &=& \left( \frac {X(t)}{n} \right)^k - \left( \frac {X(t)-R(t)}{n} \right)^k \\
\Prob( \mathcal{A}^c_{t+1} ) &=& \left( \frac {X(t)-R(t)}{n} \right)^k - \left( \frac {R(t)}{n} \right)^k \\
\Prob( \mathcal{A}^d_{t+1} ) &=& \left( \frac {R(t)}{n} \right)^k.
\end{eqnarray*}
% The set of vertices is partitioned into 
% unsaturated vertices (the rest is of size X(t))
% red vertices (the rest is of size X(t)-R(t))
% uncoloured vertices (the rest is of size R(t))
% green vertices 
Since $X(t)$ increases by 2 when $\mathcal{A}^a_{t+1}$ or $\mathcal{A}^b_{t+1}$ occur, and does not change otherwise, we get that
\begin{equation}
\E [ X(t+1) - X(t) ~|~ H_t ] = 2 \cdot \left( 1 - \left( \frac {X(t)-R(t)}{n} \right)^k \right) + \bigo (1/n). \label{eq:trendX_perfet_matching}
\end{equation}
The term $\bigo (1/n)$ corresponds to the probability that $v_{t+1}$ is the same as $u_{t+1}$ (in Case~(a)) or the same as $y$ (in Case~(b)).

The analysis of $R(t)$ is slightly more complicated. If $\mathcal{A}^a_{t+1}$ occurs, then two vertices in $U_t$ become saturated after the augmentation. Since the endpoints of the set of green edges (those with circles) are uniformly distributed in $U_t$, the expected number of green edges incident with at least one of the two vertices is equal to $2R(t)/(n-X(t))$. The other endpoints of these green edges become uncoloured from green after the augmentation which, in turn, forces their mates to become uncoloured from red. If $\mathcal{A}^b_{t+1}$ occurs, then the situation is similar, except that $u_{t+1}$ is first uncoloured from red and its mate is uncoloured from green. If $\mathcal{A}^c_{t+1}$ occurs, then a new green vertex is created which, in turn, makes its mate red. Finally, If $\mathcal{A}^d_{t+1}$ occurs, then there is no change to R(t). It follows that
\begin{eqnarray}
\E [ R(t+1) - R(t) ~|~ H_t ] &=& \Prob( \mathcal{A}^a_{t+1} ) \cdot \left( - \frac {2R(t)}{n-X(t)} \right) + \Prob( \mathcal{A}^b_{t+1} ) \cdot \left( -1 - \frac {2(R(t)-1)}{n-X(t)} \right) \nonumber \\
&& + \Prob( \mathcal{A}^c_{t+1} ) + \bigo (1/n) \nonumber \\
&=& - \left( \Prob( \mathcal{A}^a_{t+1} ) + \Prob( \mathcal{A}^b_{t+1} ) \right) \cdot \frac {2R(t)}{n-X(t)} - \Prob( \mathcal{A}^b_{t+1} ) + \Prob( \mathcal{A}^c_{t+1} ) + \bigo (1/n) \nonumber \\
&=& - \left( 1 - \left( \frac {X(t)-R(t)}{n} \right)^k \right) \cdot \frac {2R(t)}{n-X(t)} \nonumber \\
&& - \left( \frac {X(t)}{n} \right)^k + 2 \left( \frac {X(t)-R(t)}{n} \right)^k - \left( \frac {R(t)}{n} \right)^k + \bigo (1/n). \label{eq:trendR_perfet_matching}
\end{eqnarray}
By writing $x(s) = X(sn)/n$ and $r(s) = R(sn)/n$, we have that
\begin{eqnarray}
x' &=& 2 (1-(x-r)^k), \nonumber \\
r' &=& \frac {-2(1-(x-r)^k)r}{1-x} - x^k +2 (x-r)^k - r^k, \label{eq:DEs_perfect_matching}
\end{eqnarray}
with the initial conditions $x(0)=0$ and $r(0)=0$. 

Let us now check that the assumptions of the DEs method are satisfied. Let $\eps > 0$ be an arbitrarily small constant. Note that the right hand sides of~(\ref{eq:DEs_perfect_matching}) are continuous, bounded, and Lipschitz in the connected open set 
$$
\scr{D} = \{ (s,x,r) : -\eps < s < 2, -\eps < x < 1-\eps/3, -\eps < r < 1+\eps \},
$$
which contains the point $(0, x(0), r(0)) = (0, 0, 0)$. There is no error in the `Initial Condition' so it holds with any $\lambda = \Omega(\delta)$. The `Trend Hypothesis' holds with $\delta = \bigo(1/n)$ (see~(\ref{eq:trendX_perfet_matching}) and~(\ref{eq:trendR_perfet_matching})) so any $\lambda = \Omega(1/n)$ works. Trivially, $|X(t+1)-X(t)| \le 2$ for every $t \le T_{\scr{D}}$. To estimate $|R(t+1)-R(t)|$, first note that for any unsaturated vertex, the expected number of green vertices that are adjacent to it is equal to $R(t) / |U_t| = R(t) / (n - X(t)) \le 1 / (\eps / 3) = \bigo(1)$. It follows from Chernoff's bound that with probability $\bigo(n^{-2})$, for any $1 \le t \le T_{\scr{D}} \le 2n$ we have $|R(t+1)-R(t)| \ge (\log n)^2$. Hence, the `Boundedness Hypothesis' holds with $\gamma = \bigo(n^{-2})$ and $\beta= (\log n)^2$. It follows from Theorem~\ref{thm:differential_equation_method}, applied with $\lambda = n^{-1/4}$, $\gamma = \bigo(n^{-2})$ and $\beta= (\log n)^2$, that the differential equations~(\ref{eq:DEs_perfect_matching}) with the given initial conditions have a unique solution that can be extended arbitrarily close to the boundary of $\scr{D}$ and, more importantly, a.a.s.\ for every $t$ such that $t/n < \sigma$, where $\sigma$ is the supremum of $s$ where $x(s) \le 1 - \eps / 2$ and $s < 2$, 
$$
\max \Big\{ |X(t) - x(t/n) n|, |R(t) - r(t/n) n| \Big\} = \bigo ( \lambda n ) = o(n).
$$
Numerical calculations show that $x(s)$ reaches $1-\eps/2$ before $s$ reaches $2$. This gives us a bound (that holds a.a.s.) for the number of steps for the process to reach at most $\eps n$ unsaturated vertices and the clean-up algorithm can deal with the rest. 
\end{proof}

%%%%%%%%%%%%%%%%%%%%%%%%%%%%%%%%%%%%%%%%%%%%%%%%%
\subsection{Large value of $k$}\label{sec:pm_large}

Let us show that $\tau_{\texttt{PM}}(k) \to 1/2$ as $k \to \infty$.

\begin{theorem}\label{thm:largek_pm}
The following bounds hold:
$$
\frac {1}{2} \le \tau_{\texttt{PM}}(k) \le \frac {1}{2} \Big( 1 + \bigo( \sqrt{ \log k / k} ) \Big).
$$
\end{theorem}

\begin{proof}
The lower bound is trivial. Deterministically, a perfect matching cannot be created in less than $n/2$ rounds.

Let us now move to the upper bound. During the first phase that lasts for $n/2$ steps, we will consider the following greedy algorithm. If at least one square lands on an unsaturated vertex, then a partial matching is extended; otherwise, an edge that is created at this step is simply ignored and will not be used in the process of constructing a perfect matching.

Let $Y(t)$ be the number of unsaturated vertices at time $t$. We will show that a.a.s.\ after $n/2$ steps, all but at most $\log k / k$ fraction of vertices are saturated, that is, a.a.s.\ $Y(n/2) \le n \log k / k$. For a contradiction, suppose that $Y(n/2) > n \log k / k$. It implies that at any step $t$, $1 \le t \le n/2$, $Y(t) \ge Y(n/2) > n \log k / k$ and so a partial matching cannot be extended at time $t$ with probability 
$$
\left( 1 - \frac {Y(t)}{n} \right)^k \le \exp \left( - \frac {kY(t)}{n} \right) < \exp( - \log k ) = 1/k.
$$
Hence, we expect at most $n/(2k)$ steps failing to extend the matching and so a.a.s.\ at most $n/(2k) + o(n)$ steps do that by Chernoff's bound. We get that a.a.s.\
$$
Y(n/2) \le 2 \cdot \left( \frac {n}{2k} + o(n) \right) = n/k + o(n) \le n \log k / k
$$
assuming that $k \ge 3$ (which we may, since we aim for a result that holds for $k$ large enough). The desired contradiction implies that a.a.s.\ at the end of the first phase, there are at most $n \log k / k$ unsaturated vertices.

During the second phase, the clean-up algorithm analyzed in~\cite{gao2022perfect} can be used to finish the job and to saturate the remaining $\eps = \eps(k) = \log k / k$ fraction of vertices. It was proved in~\cite{gao2022perfect} that a.a.s.\ this algorithm takes at most $100\sqrt{\eps} n = \bigo ( n \sqrt{ \log k / k } )$ steps, which finishes the proof of the theorem.
\end{proof}

%%%%%%%%%%%%%%%%%%%%%%%%%%%%%%%%%%%%%%%%%%%%%%%%%
\section{Hamiltonian Cycles}\label{sec:cycles}
%%%%%%%%%%%%%%%%%%%%%%%%%%%%%%%%%%%%%%%%%%%%%%%%%

In this section, we concentrate on another classical property, namely, the property of having a Hamiltonian cycle, which we denote by ${\tt HAM}$. It is known that a.a.s.\ the $3$-out process we discussed in the previous section is Hamiltonian~\cite{bohman2009hamilton}. As already mentioned earlier, the semi-random process can be coupled with the $\ell$-out process~\cite{beneliezer2019semirandom} (for any $\ell \in \Nn$) and so we get that $\tau_{\tt HAM} \le 3$. A new upper bound was obtained in~\cite{gao2020hamilton} in terms of an optimal solution to an optimization problem whose value is believed to be $2.61135$ by numerical support. 

The upper bound on $\tau_{\tt HAM}$ of $3$ obtained by simulating the $3$-out process is \emph{non-adaptive}. That is, the strategy does \textit{not} depend on the history of the semi-random process. The above mentioned improvement proposed in~\cite{gao2020hamilton} uses an adaptive strategy but in a weak sense. The strategy consists of 4 phases, each lasting a linear number of rounds, and the strategy is adjusted \emph{only} at the end of each phase (for example, the player might identify vertices of low degree, and then focus on putting circles on them during the next phase). 

In~\cite{gao2022fully}, a fully adaptive strategy was proposed that pays attention to the graph $G_t$ and the position of $u_t$ for every single step $t$. As expected, such a strategy creates a Hamiltonian cycle substantially faster than the weakly adaptive or non-adaptive strategies, and it allows to improve the upper bound from $2.61135$ to $2.01678$. One more trick was observed recently which further improves the upper bound to $1.84887$~\cite{frieze2022hamilton}. After combining all the ideas together, the currently best upper bound is equal to $1.81701$~\cite{newpreprint}. In this paper, we adjust a slightly easier version of the algorithm from~\cite{frieze2022hamilton} to deal with the $k$-process and present the corresponding upper bounds (see Subsection~\ref{sec:ham_upperbound}).

Let us now move to the lower bounds. As observed in the initial paper introducing the semi-random process~\cite{beneliezer2019semirandom}, if $G_t$ has a Hamiltonian cycle, then $G_t$ has minimum degree at least 2. Thus, $\tau_{\tt HAM} \ge  \tau_{\scr{P}_2}  = \ln 2+ \ln(1+\ln2) \ge 1.21973$, where $\scr{P}_2$ corresponds to the property of having the minimum degree at least $2$---see Section~\ref{sec:degree}. In~\cite{gao2020hamilton}, the lower bound mentioned above was shown to not be tight. However, the lower bound was only increased by $\eps = 10^{-8}$ and so numerically negligible. A better bound was obtained in~\cite{gao2022fully} (see also~\cite{newpreprint}) and now we know that $\tau_{\tt HAM} \ge 1.26575$. Adjusting the lower bound from~\cite{gao2022fully} seems challenging and technical so we only report trivial lower bounds using the results from Section~\ref{sec:degree}:  $\tau_{\texttt{HAM}}(k) \ge \tau_{\scr{P}_2}(k)$. The gap between the upper and lower bounds gets small as $k$ gets large. In fact, $\tau_{\texttt{HAM}}(k) \to 1$ as $k \to \infty$ (see Subsection~\ref{sec:ham_large}). 

\begin{table}[htp]
\caption{Hamilton Cycles---numerical upper and lower bounds of $\tau_{\texttt{HAM}}(k)$ for $1 \le k \le 10$. Stronger upper and lower bounds for $k=1$ follow from~\cite{newpreprint} and~\cite{gao2022fully} respectively.}
\begin{center}
\begin{tabular}{|c|c|c|c|c|c|c|}
\hline
& lower bound & upper bound & & lower bound & upper bound \\
\hline 
$k=1$ & 1.21974 (1.26575) & 1.87230 (1.81701) & $k=6$ & 1.05075 & 1.13325 \\
$k=2$ & 1.12498 & 1.39618 & $k=7$ & 1.04426 & 1.11534 \\
$k=3$ & 1.09081 & 1.26077 & $k=8$ & 1.03924 & 1.10180 \\
$k=4$ & 1.07184 & 1.19615 & $k=9$ & 1.03525 & 1.09115 \\
$k=5$ & 1.05947 & 1.15827 & $k=10$ & 1.03199 & 1.08254 \\
\hline
\end{tabular}
\end{center}
\label{tab:cycles}
\end{table}

%%%%%%%%%%%%%%%%%%%%%%%%%%%%%%%%%%%%%%%%%%%%%%%%%
\subsection{Upper bound for $\tau_{\texttt{HAM}}(k)$}\label{sec:ham_upperbound}

In this subsection, we adjust a slightly easier version of the algorithm from~\cite{frieze2022hamilton} to deal with the $k$-process for any $k \in \Nn$. For $k=1$, it yields a bound of $1.87230$ which is slightly worse than the one reported in~\cite{frieze2022hamilton} ($1.84887$) and in~\cite{newpreprint} ($1.81696$) but is easier to analyze. For $k \ge 2$ the difference would be even smaller, but with substantially larger effort one may do it. 

The algorithm builds a path that eventually becomes a Hamiltonian path and then it is turned into a Hamiltonian cycle. Let $X(t)$ be the number of vertices that belong to the path $P_t$ that is present at time $t$. Some of the vertices outside of $P_t$ will be matched with each other and will form a \textbf{matching}. Let $Y(t)$ be the number of vertices outside of $P_t$ that are \textbf{matched}. The remaining vertices (not on the path $P_t$ nor matched) are \textbf{unsaturated}. Let $U(t)$ be the number of unsaturated vertices. 

It is convenient to colour some of our vertices and edges red. Vertices on the path $P_t$ are \textbf{red} if they are adjacent to precisely one red edge of $G_t$ (this edge will not belong to $P_t$). Let $R(t)$ be the number of red vertices. It will be useful to maintain the property that no two red vertices are at the path distance less than 3 from each other. Assume then that this property is satisfied at time $t$. Clearly, there are $2R(t)+\bigo(1)$ vertices of $P_t$ that are at distance 1 from the set of red vertices; we colour such vertices \textbf{green}. (Note that it is possible that one or both of the endpoints of the path are red and such vertices have only one neighbour on the path. This explains additional $\bigo(1)$ term.) Moreover, there are at most $2R(t)$ vertices that are at distance 2 from the set of red vertices and are not green (nor red, of course, as no two red vertices are at distance 2 from each other); we call them \textbf{useless}. The reason to introduce useless vertices is to make sure that each red vertex (except possibly two red vertices at the end of the path) is surrounded by unique two green vertices. If, at some point of the process, the square lands on a green vertex, then the player can extend the path by placing a circle at the endpoint of the associated red edge that does not belong to the path. To simplify the analysis, we arbitrarily select more vertices on the path that are not red nor green and call them useless so that there are $2R(t)+\bigo(1)$ useless vertices. Vertices on $P_t$ that are not coloured nor useless are called \textbf{permissible}. 

Note that in each step of the process the set of vertices is partitioned into 6 sets: red vertices, green vertices, useless vertices, permissible vertices, matched vertices, and unsaturated vertices. We will track the length of the path $P_t$ (random variable $X(t)$), the number of red vertices (random variable $R(t)$) and the number of matched vertices (random variable $Y(t)$). By design, the number of green and useless vertices are both equal to $2R(t)+\bigo(1)$ so there is no need to track them. Similarly, the number of permissible vertices is equal to $X(t) - 5R(t)+\bigo(1)$. Finally, the number of unsaturated vertices is equal to $n - X(t) - Y(t)$.

\medskip

Suppose that at time $t$, $k$ squares land on vertices $u^1_t, \ldots, u^k_t$. We consider a few cases. The algorithm performs the first case that holds.
\begin{itemize}
\item [Case~(a):] At least one square lands on an unsaturated vertex. We arbitrarily select one of them to be $u_t$ and let $v_t$ be a uniformly random unsaturated vertex. If $u_t = v_t$, then we do nothing; otherwise, we extend the partial matching by adding the edge $u_t v_t$ we just created to it. 
\item [Case~(b):] At least one square lands on a matched vertex. We arbitrarily select one of the matched vertices to be $u_t$ and let $v_t$ be one of the two endpoints of the path. We greedily extend $P_{t-1}$ by adding $v_t u_t$ and the edge containing $u_t$ from the matching to the path. If some red vertex $x$ is adjacent to either of the two absorbed vertices by a red edge, then we uncolour this red edge and uncolour $x$. This, in turn, uncolours green neighbours of $x$.
\item [Case~(c):] At least one square lands on a green vertex. We arbitrarily select one of these green vertices to be $u_t$ and let $y$ be the unique red neighbour of $u_t$. We augment $P_{t-1}$ via the unique red edge $yz$. If $z$ is unsaturated (sub-case~(c')), then we let $v_t = z$, add edges $u_t v_t = u_t z$, $v_t y = z y$, and remove edge $u_t y$ from $P_{t-1}$ to form $P_t$. On the other hand, if $z$ is matched to vertex $q$ (sub-case~(c'')), then we let $v_t = q$, add edges $u_t v_t = u_t q$, $q z = v_t z$, $z y$, and remove edge $u_t y$. If some red vertex $x$ is adjacent to the absorbed vertex $z$ (or the absorbed vertices $z$ and $q$ in the second sub-case) by a red edge, then we uncolour this red edge and uncolour $x$. As before, this uncolours green neighbours of $x$.
\item [Case~(d):] At least one square lands on a permissible vertex. We arbitrarily select one of these vertices to be $u_t$. Then, we choose $v_t$ uniformly at random amongst matched and unsaturated vertices, and colour $u_t v_t$ red. This case creates one red vertex, namely, vertex $u_t$, and two green vertices (or one if $u_t$ is one of the endpoints of the path). 
\item [Case~(e):] All squares land on useless or red vertices. In this case, we choose $v_t$ arbitrarily and interpret the algorithm as passing on this round, meaning the edge $u_t v_t$ will not be used to construct a Hamiltonian cycle. 
\end{itemize}

The analysis of the above algorithm is the main ingredient of the proof of the next result. Let us note that one may consider different orders of the above five cases yielding $5! = 120$ greedy algorithms. We selected the order that gives the strongest upper bound. 

\begin{theorem}
Let $k \in \Nn$. Then, $\tau_{\texttt{HAM}}(k) \le u_{k}$, where $u_{k} \in [1,3)$ is a constant that is derived from a system of differential equations. The numerical bounds for $1 \le k \le 10$ are presented in Table~\ref{tab:cycles}.
\end{theorem}

\begin{proof}
Let $\mathcal{A}^i_{t+1}$ be the event that Case~(i) occurred at step $t+1$. It follows that
\begin{eqnarray*}
\Prob( \mathcal{A}^a_{t+1} ) &=& 1 - \left( \frac {X(t)+Y(t)}{n} \right)^k \\
\Prob( \mathcal{A}^b_{t+1} ) &=& \left( \frac {X(t)+Y(t)}{n} \right)^k - \left( \frac {X(t)}{n} \right)^k \\
\Prob( \mathcal{A}^c_{t+1} ) &=& \left( \frac {X(t)}{n} \right)^k - \left( \frac {X(t)-2R(t)+\bigo(1)}{n} \right)^k \\
\Prob( \mathcal{A}^d_{t+1} ) &=& \left( \frac {X(t)-2R(t)+\bigo(1)}{n} \right)^k - \left( \frac {3R(t) +\bigo(1)}{n} \right)^k \\
\Prob( \mathcal{A}^e_{t+1} ) &=& \left( \frac {3R(t)+\bigo(1)}{n} \right)^k.
\end{eqnarray*}

We first need to estimate the expected change in the three random variables we track. Let us denote $H_t = ( X(i), R(i), Y(i))_{0 \le i \le t}$. Note that $H_t$ does \emph{not} encompass the entire history of the random process after $t$ rounds (that is, $G_0, \ldots ,G_t$). This deferred information exposure permits a tractable analysis of the random positioning of $v_t$ when $u_t$ is red. In particular, as we only expose $R(t)$ instead of the exact locations of the red edges, their endpoints which are not on the path are random vertices in $[n] \setminus V(P_t)$. Similarly, as we only expose $Y(t)$ instead of the exact locations of the edges that form a matching, these edges have the same distribution (conditional on $H_t$) as first exposing the set of vertices in $[n] \setminus V(P_t)$, then uniformly selecting a subset of vertices in $[n]\setminus V(P_t)$ of cardinality $Y(t)$, and then finally taking a uniformly random perfect matching over the $Y(t)$ vertices (that is, pair the $Y(t)$ vertices into $Y(t)/2$ disjoint edges of the matching).

\medskip

We observe the following expected difference equations. Let us start from $X(t)$, which is the easiest to deal with. $X(t)$ changes only when case~(b) or case~(c) occurs; it increases by 2 in case~(b) and increases by 1 or 2 in case~(c). Conditioning on case~(c) occurring, since the endpoint of the red edge we augment via is a random vertex in $[n] \setminus V(P_t)$, sub-case~(c') occurs with probability $(n-X(t)-Y(t))/(n-X(t))$ and sub-case~(c'') occurs with probability $Y(t)/(n-X(t))$---we expect to absorb $1 + Y(t)/(n-X(t))$ vertices. We get that
\begin{equation}
\E [X(t+1) - X(t) \mid H_t] = 2 \cdot \Prob( \mathcal{A}^b_{t+1} ) + \left( 1 + \frac {Y(t)}{n-X(t)} \right) \cdot \Prob( \mathcal{A}^c_{t+1} ). \label{diff-X}
\end{equation}
Investigating the behaviour of $Y(t)$ is also relatively easy to do. $Y(t)$ increases by 2 in case~(a) and decreases by 2 in case~(b). In case~(c), it may decrease by 2 but only when the endpoint of a red edge is matched (sub-case~(c'')). We get that
\begin{equation}
\E [Y(t+1) - Y(t) \mid H_t] = 2 \cdot \Prob( \mathcal{A}^a_{t+1} ) - 2 \cdot \Prob( \mathcal{A}^b_{t+1} ) - 2 \cdot \Prob( \mathcal{A}^c_{t+1} ) \cdot \frac {Y(t)}{n-X(t)}. \label{diff-Y}
\end{equation}
The most challenging part is to understand the behaviour of $R(t)$. The contribution to the expected change comes from two sources. In case~(c) we augment via a red edge so one red vertex gets uncoloured and in case~(d) we create one red vertex. The second source is associated with the fact that when one or two vertices get absorbed into the path all red edges incident to them get uncoloured. The expected number of vertices that get absorbed is already computed in~(\ref{diff-X}). Each vertex that gets absorbed uncolours $(R(t)+\bigo(1))/(n-X(t))$ red vertices. We get that
\begin{eqnarray}
\E [R(t+1) - R(t) \mid H_t] &=& - \Prob( \mathcal{A}^c_{t+1} ) + \Prob( \mathcal{A}^d_{t+1} ) \nonumber \\
&& - \left( 2 \cdot \Prob( \mathcal{A}^b_{t+1} ) + \left( 1 + \frac {Y(t)}{n-X(t)} \right) \cdot \Prob( \mathcal{A}^c_{t+1} ) \right) \cdot \frac {R(t)+\bigo(1)}{n - X(t)} \nonumber \\
&=& - \frac {2 R(t)}{n - X(t)} \cdot \Prob( \mathcal{A}^b_{t+1} ) \label{diff-R} \\
&& - \left( \frac {(n-X(t)+Y(t)) \cdot (R(t)+\bigo(1))}{(n-X(t))^2} + 1 \right) \cdot \Prob( \mathcal{A}^c_{t+1} ) + \Prob( \mathcal{A}^d_{t+1} ). \nonumber 
\end{eqnarray}

After rescaling, (\ref{diff-X}), (\ref{diff-Y}), and (\ref{diff-R}), we get the following set of DEs:
\begin{eqnarray}
x' &=& 2 \Big( (x+y)^k - x^k \Big) + \left( 1 + \frac {y}{1-x} \right) \Big( x^k - (x-2r)^k \Big) \nonumber \\
y' &=& 2 \Big( 1 - (x+y)^k \Big) - 2 \Big( (x+y)^k - x^k \Big) - 2 \Big( x^k - (x-2r)^k \Big) \frac {y}{1-x} \nonumber \\
r' &=& - \frac {2r}{1-x} \Big( (x+y)^k - x^k \Big) - \left( \frac { (1-x+y)r }{ (1-x)^2 } + 1 \right) \Big( x^k - (x-2r)^k \Big) \label{eq:DEs_ham}\nonumber \\
&& + \Big( (x-2r)^k - (3r)^k \Big), 
\end{eqnarray}
with the initial conditions $x(0)=0$, $y(0)=0$, and $r(0)=0$.

As usual, we need to check that the assumptions of the DEs method are satisfied. Let $\eps > 0$ be an arbitrarily small constant. Initially, $X(0)=Y(0)=R(0)=0$ so the `Initial Condition' trivially holds. The right hand sides of all equations in~(\ref{eq:DEs_ham}) are continuous, bounded, and Lipschitz in the connected open set 
$$
\scr{D}_\eps = \{ (s, x, y, r) : - 1 < s < 3, -1 < x < 1 -\eps, -1 < y,r < 2 \},
$$
which contains the point $(0,0,0,0)$. (Note that we need to restrict the interval for $x$ due to a singularity point $x=1$.) Define 
\[
T_{\scr{D}_{\eps}}=\min\{t\ge 0: \ (t/n, X(t)/n, Y(t)/n, R(t)/n)\notin \scr{D}_{\eps}\}.
\]
The `Trend Hypothesis' holds with $\delta = \bigo( 1/n )$. The `Boundedness Hypothesis' requires more investigation. Random variables $X(t)$ and $Y(t)$ can change by at most 2 in each round. To estimate the maximum change for the random variable $R(t)$, we need to upper bound the number of red edges adjacent to any unsaturated or matched vertex $v$. Observe that at any step $t \le 3n$, since we have assumed that there are at least $\eps n$ unsaturated or matched vertices, the number of red edges adjacent to $v$ is stochastically upper bounded by $\Bin(3n, 1/(\eps n))$ with expectation $3/\eps$. It follows immediately from Chernoff's bound~(\ref{chern1}) that with probability $1-\bigo(n^{-3})$, the number of red vertices adjacent to $v$ is at most $\beta = \bigo( \log n)$. Hence, the `Boundedness Hypothesis' holds with probability at least $1- \gamma$ with $\gamma = \bigo(n^{-1})$ by taking the union bound over all $3n^2$ vertices and steps.

We conclude, based on Theorem~\ref{thm:differential_equation_method}, that for every $\tau >0$, a.a.s.\ for any $0 \le t \le (\sigma(\eps)-\tau)n$,
$$
\max \Big\{ |X(t) - x(t/n) n|, |Y(t) - y(t/n) n|, |R(t) - r(t/n) n| \Big\} = \bigo ( \lambda n ) = o(n),
$$
where $x, y, r$ are the unique solutions of the above DEs satisfying the desired initial conditions, and $\sigma(\eps)$ is the supremum of $s$ to which the solution can be extended before reaching the boundary of $T_{\scr{D}_{\eps}}$. As $\scr{D}_{\eps}\subseteq \scr{D}_{\eps'}$ for every $\eps>\eps'>0$, $\sigma(\eps)$ is monotonely nondecreasing as $\eps\to 0$. Thus, 
$$
u_k :=\lim_{\eps \to 0+} \sigma(\eps) 
$$
exists. It is obvious that $|Y(t)/n|$ and $|R(t)/n|$ are both bounded by 1 for all $t$ and thus, when $t/n$ approaches $u_k$, either $X(t)/n$ approaches 1 or $t/n$ approaches 3. If follows that a.a.s.\ either $X(t)>(1-\eps)n$ for all $t \ge (u_k-\delta)n$ or $u_k=3$. The above DEs do not have an analytical solution but numerical solutions show that $u_k \le u_1 < 1.87230$. Hence, by the end of the execution of the algorithm, there are $\eps n$ unsaturated or matched vertices remaining for some $\eps = o(1)$. 

The clean-up algorithm analyzed in~\cite{gao2022fully} (see also~\cite{newpreprint}) absorbs the remaining $\eps n = o(n)$ vertices into the path to form a Hamiltonian path, after which a Hamiltonian cycle is constructed. The whole procedure takes $\bigo (\sqrt{\eps}n + n^{3/4}\log^2 n) = o(n)$ further steps, which finishes the proof of the theorem. 
\end{proof}

%\hc{
%\begin{eqnarray*}
%x' &=&  2\times\Prob( \mathcal{A}^b_{t+1} )  + \frac{v_{1} + 2v_{2}}{1-x}\times \Prob( \mathcal{A}^c_{t+1} )\\
%v'_{1} &=& -2 \times \Prob( \mathcal{A}^a_{t+1} ) -\frac{v_1}{1-x} \times \Prob( \mathcal{A}^c_{t+1} ) \\
%v'_{2} &=& -2\times \Prob( \mathcal{A}^b_{t+1} ) +2 \times \Prob( \mathcal{A}^a_{t+1} )  -\frac{2v_2}{1-x} \times \Prob( \mathcal{A}^c_{t+1} ) \\
%r' &=& \Prob( \mathcal{A}^d_{t+1} )  - (1+r\frac{(v_1+2v_2)}{(1-x)^2}) \times \Prob( \mathcal{A}^c_{t+1} )  - \frac{2r}{1-x} \times \Prob( \mathcal{A}^b_{t+1} ) .
%\end{eqnarray*}
%}

%%%%%%%%%%%%%%%%%%%%%%%%%%%%%%%%%%%%%%%%%%%%%%%%%
\subsection{Large value of $k$}\label{sec:ham_large}

In this subsection, we show that $\tau_{\texttt{HAM}}(k) \to 1$ as $k \to \infty$. 

\begin{theorem}
The following bounds hold:
$$
1 \le \tau_{\texttt{HAM}}(k) \le 1 + \bigo( \sqrt{ \log k / k} ).
$$
\end{theorem}
\begin{proof}
The proof is almost the same as the proof of Theorem~\ref{thm:largek_pm}. The lower bound is trivial: one cannot create a Hamilton cycle in less than $n$ rounds. 

As before, during the first phase that lasts for $n$ steps, a greedy algorithm is used that extends a path whenever at least one square lands on a vertex that is not on the path. At the end of this phase, a.a.s.\ a path of length at least $n - n \log k / k$ is created.

During the second phase, another clean-up algorithm can be used (analyzed in~\cite{gao2022fully}) to finish the job and to absorb the remaining $\eps = \eps(k) = \log k / k$ fraction of vertices. It was proved in~\cite{gao2022fully} that a.a.s.\ this algorithm takes $\bigo ( \sqrt{\eps} n) = \bigo ( n \sqrt{ \log k / k } )$ steps, which finishes the proof of the theorem.
\end{proof}

\bibliographystyle{plain}

\bibliography{refs.bib}

\end{document}